\documentclass[12pt]{amsart}
\usepackage{amscd}
\usepackage{amsmath}
\usepackage{amssymb}
\usepackage{color}
\usepackage[all]{xy}

\def\frk{\frak}               

\def\Phi{{\frk n}}
\def\Phi{{\frk N}}
\def\opn#1#2{\def#1{\operatorname{#2}}} 
\opn\chara{char} \opn\length{\ell} \opn\pd{pd} \opn\rk{rk}
\opn\projdim{proj\,dim} \opn\injdim{inj\,dim} \opn\rank{rank}
\opn\depth{depth} \opn\sdepth{sdepth} \opn\fdepth{fdepth}
\opn\grade{grade} \opn\height{height} \opn\embdim{emb\,dim}
\opn\codim{codim}  \opn\min{min} \opn\max{max}

\opn\Tr{Tr} \opn\bigrank{big\,rank}
\opn\superheight{superheight}\opn\lcm{lcm}
\opn\trdeg{tr\,deg}
\opn\reg{reg} \opn\lreg{lreg} \opn\ini{in} \opn\lpd{lpd}
\opn\size{size}
\opn\div{div} \opn\Div{Div} \opn\cl{cl} \opn\Cl{Cl}
\opn\Spec{Spec} \opn\Supp{Supp} \opn\supp{supp} \opn\Sing{Sing}
\opn\Ass{Ass} \opn\Min{Min}
\opn\Ann{Ann} \opn\Rad{Rad} \opn\Soc{Soc}
\opn\Im{Im} \opn\Ker{Ker} \opn\Coker{Coker} \opn\Am{Am}
\opn\Hom{Hom} \opn\Tor{Tor} \opn\Ext{Ext} \opn\End{End}
\opn\Aut{Aut} \opn\id{id}  \opn\deg{deg}

\opn\nat{nat}
\opn\pff{pf}
\opn\Pf{Pf} \opn\GL{GL} \opn\SL{SL} \opn\mod{mod} \opn\ord{ord}
\opn\Gin{Gin} \opn\Hilb{Hilb}
\opn\aff{aff} \opn\con{conv} \opn\relint{relint} \opn\st{st}
\opn\lk{lk} \opn\cn{cn} \opn\core{core} \opn\vol{vol}
\opn\link{link} \opn\star{star}
\opn\gr{gr}

\def\pot#1#2{#1[\kern-0.28ex[#2]\kern-0.28ex]}

\opn\dirlim{\underrightarrow{\lim}}
\opn\inivlim{\underleftarrow{\lim}}
\let\Dirsum=\bigoplus

\let\to=\rightarrow

\def\Implies{\ifmmode\Longrightarrow \else
        \unskip${}\Longrightarrow{}$\ignorespaces\fi}
\def\implies{\ifmmode\Rightarrow \else
        \unskip${}\Rightarrow{}$\ignorespaces\fi}
\def\iff{\ifmmode\Longleftrightarrow \else
        \unskip${}\Longleftrightarrow{}$\ignorespaces\fi}

\let\:=\colon
\newtheorem{Theorem}{Theorem}[section]
\newtheorem{Lemma}[Theorem]{Lemma}

\newtheorem{Proposition}[Theorem]{Proposition}
\newtheorem{Remark}[Theorem]{Remark}

\newtheorem{Conjecture}[Theorem]{Conjecture}

\let\epsilon\varepsilon
\let\phi=\varphi
\let\kappa=\varkappa
\def\qed{\ifhmode\textqed\fi
      \ifmmode\ifinner\quad\qedsymbol\else\dispqed\fi\fi}
\def\textqed{\unskip\nobreak\penalty50
       \hskip2em\hbox{}\nobreak\hfil\qedsymbol
       \parfillskip=0pt \finalhyphendemerits=0}
\def\dispqed{\rlap{\qquad\qedsymbol}}

\opn\dis{dis}
\def\pnt{{\raise0.5mm\hbox{\large\bf.}}}

\opn\Lex{Lex}

\begin{document}

\title{Three generated, squarefree, monomial ideals}

\author{Dorin Popescu and Andrei Zarojanu}

\date{}

\pagestyle{myheadings}
\markboth{Dorin Popescu and Andrei Zarojanu }{Three generated, squarefree, monomial ideals }

\maketitle 

\begin{abstract}
    Let $I\supsetneq J$ be  two  squarefree monomial ideals of a polynomial algebra over a field generated in degree $\geq d$, resp. $\geq d+1$ .  Suppose that  $I$ is generated  by  three monomials of degrees $d$. If  the Stanley depth of $I/J$ is $\leq d+1$ then the usual depth of $I/J$ is $\leq d+1$ too.
\end{abstract}

\begin{quotation}

   Se\-con\-dary      Secondary 13F20, 13F55,
13P10.
\end{quotation}


\thispagestyle{empty}

\section{Introduction}

  Let  $S=K[x_1,\ldots,x_n]$, $n\in {\bf N}$, be a polynomial ring over a field $K$. Let $I\supsetneq J$  be two   squarefree monomial ideals of $S$ and  $u\in I \setminus J$ a monomial in $I/J$.
  For $Z\subset \{x_1,\ldots ,x_n\}$  with $(J:u)\cap K[Z]=0$, let $uK[Z]$ be the linear $K$-subspace of $I/J$ generated by the elements $uf$, $f\in K[Z]$.  A  presentation of $I/J$ as a finite direct sum of such spaces ${\mathcal D}:\
I/J=\Dirsum_{i=1}^ru_iK[Z_i]$ is called a {\em Stanley decomposition} of $I/J$. Set $\sdepth
(\mathcal{D}):=\min\{|Z_i|:i=1,\ldots,r\}$ and
\[
\sdepth\ I/J :=\max\{\sdepth \ ({\mathcal D}):\; {\mathcal D}\; \text{is a
Stanley decomposition of}\;  I/J \}.
\]

Stanley's Conjecture says that the {\em Stanley depth}  $\sdepth_S I/J\geq \depth_S I/J$.
 The Stanley  depth of $I/J$ is a combinatorial invariant and  does not depend on the characteristic of the field $K$. If $J=0$ then this conjecture holds for $n\leq 5$ by \cite{P2}, or when $I$ is an intersection of four monomial prime ideals by \cite{AP}, \cite{P3}, or   an intersection of three monomial primary ideals by \cite{Z}, or a monomial almost complete intersection by \cite{Ci}. The Stanley depth and the Stanley's Conjecture are similarly given when $I,J$ are not squarefree. In the non squarefree monomial ideals a useful inequality is $\sdepth I\leq \sdepth \sqrt{I}$ (see \cite[Theorem 2.1]{Is}).

Suppose that $I$ is generated by squarefree monomials of degrees $\geq d$   for some positive integer $d$. We may assume either that $J=0$, or $J$ is generated in degrees $\geq d+1$ after  a multigraded isomorphism.
We have $\depth_S I\geq d$ by \cite[Proposition 3.1]{HVZ} and it follows $\depth_S I/J\geq d$ (see  \cite[Lemma 1.1]{P}).
  Depth of $I/J$ is a homological invariant and  depends on the characteristic of the field $K$. 
  The Stanley decompositions of $S/J$ corresponds bijectively to partitions into intervals of the simplicial complex whose Stanley-Reisner ring is $S/J$. If Stanley's Conjecture holds then the simplicial complexes are partitionable (see \cite{HJY}). Using this idea an equivalent definition of Stanley's depth of $I/J$ was given in \cite{HVZ}.

   Let $P_{I\setminus J}$  be the poset of all squarefree monomials of $I\setminus J$  with the order given by the divisibility. Let ${\mathcal P}$ be a partition of  $P_{I\setminus J}$ in intervals $[u,v]=\{w\in  P_{I\setminus J}: u|w, w|v\}$, let us say   $P_{I\setminus J}=\cup_i [u_i,v_i]$, the union being disjoint.
Define $\sdepth {\mathcal P}=\min_i\deg v_i$. Then $\sdepth_SI/J=\max_{\mathcal P} \sdepth {\mathcal P}$, where ${\mathcal P}$ runs in the set of all partitions of $P_{I\setminus J}$ (see  \cite{HVZ}, \cite{S}). 

 For more than  thirty years the Stanley Conjecture was a dream for many people working in combinatorics and commutative algebra. Many  people believe that  this conjecture holds and tried to prove directly some of its consequences. For example in this way a lower bound of depth given by Lyubeznik \cite{L} was extended by Herzog at al. \cite{HPV} for sdepth.

Some numerical upper bounds of sdepth give also upper bounds of depth, which are independent of char $K$. More precisely, write $\rho_j(I \setminus J)$ for the number of all
squarefree monomials of degrees $j$ in $I \setminus J$.
\begin{Theorem}(Popescu \cite[Theorem 1.3]{P1}) \label{t} Assume that $\depth_S(I/J) \geq t$,
where $ t$ is an integer such that $ d\leq  t < n$. If
$\rho_{t+1}(I \setminus J) < \alpha_t := \sum_{i=0}^{t-d}(-1)^{t-d+i}
\rho_{d+i}(I \setminus J)$, then
$ \depth_S(I/J) = t$ independently of the characteristic of $K$.
\end{Theorem}
The proof uses Koszul homology and is not very short. An extension is given below.
\begin{Theorem} ( Shen \cite[Theorem 2.4]{Sh})\label{t1} Assume that $\depth_S(I/J) \geq t$, where $t$ is an
integer such that $d \leq t < n$. If for some $k$ with $d + 1 \leq k \leq t + 1$ it holds
 $\rho_k(I \setminus J) <\sum_{j=d}^{k-1}(-1)^{k-j+1}{t + 1 - j\choose k - j}\rho_j(I \setminus J)$,
then $\depth_S(I/J) = t$ independently of the characteristic of $ K$.
\end{Theorem}
 Shen's proof is very short, based on a strong tool, namely the Hilbert depth considered by Bruns-Krattenhaler-Uliczka \cite{BKU} (see also \cite{U}, \cite{IM}). Thus it is important to have the right tool.

Let $r$ be the number of the squarefree monomials of degrees $d$ of $I$ and $B$ (resp. $C$) be the set of the squarefree monomials of degrees $d+1$  (resp. $d+2$) of $I\setminus J$.  Set $s=|B|$, $q=|C|$.
If $r>s$ then Theorem \ref{t}  says that $\depth_SI/J=d$, namely the minimum possible. This was  done previously in \cite{P} (the idea started  in \cite{P0}).
Moreover, Theorem \ref{t} together with Hall's marriage theorem for bipartite  graphs gives  the following:

\begin{Theorem} (Popescu \cite[Theorem 4.3]{P})\label{P} If $\sdepth_SI/J=d$ then  $\depth_SI/J=d$, that is Stanley's Conjecture holds in this case.
\end{Theorem}

The purpose of our paper is to study the next step in proving Stanley's Conjecture namely the following weaker conjecture.
\begin{Conjecture} \label{c}   Suppose that $I \subset S$ is minimally generated by some squarefree monomials $f_1,\ldots,f_k$ of degrees $d$,  and a  set $H$  of squarefree monomials of degrees $\geq d+1$. Assume that  $\sdepth_S I/J=d+1$.
Then $\depth_S I/J \leq d+1$
\end{Conjecture}
The following  theorem is a partial answer.
\begin{Theorem}   The above conjecture holds  in each of the following two cases:
\begin{enumerate}
\item{} $k= 1$,
\item{} $1<k\leq 3$, $H=\emptyset$.
\end{enumerate}
\end{Theorem}
When $k=1$ and $s\not =q+1$ the result was stated in \cite{PZ} and \cite{PZ1}. The theorem follows from Proposition \ref{l1}  and Theorems \ref{p2}, \ref{p3}.

We owe thanks to the Referee, who noticed  some  mistakes in a previous version of this paper, especially in the proof of Lemma \ref{key}.

\section{Cases $r=1$ and $d=1$}

Let $I\supsetneq J$ be  two  squarefree monomial ideals of $S$. We assume that $I$ is generated by squarefree monomials of degrees $\geq d$  for some $d\in {\bf N}$. We may suppose that either $J=0$, or  is generated by some squarefree monomials of degrees $\geq d+1$.
As above $B$ (resp. $C$) denotes the set of the squarefree monomials of degrees $d+1$  (resp. $d+2$) of $I\setminus J$.

\begin{Lemma} \label{l} Suppose that $I \subset S$ is minimally generated by some  square free monomials $\{f_1,\ldots,f_r\}$ of degrees $d$,
  and a  set $E$  of square free monomials of degrees $\geq d+1$.
  Assume that  $\sdepth_S I/J\leq d+1$ and the above Conjecture \ref{c}  holds for $k<r$ and for $k=r$, $|H|<|E|$ if $E\not=\emptyset$. If either $C\not\subset (f_2,\ldots,f_r,E)$, or
  $C\not\subset (f_1,\ldots,f_r,E\setminus \{a\})$ for some $a\in E$
  then $\depth_S I/J\leq d+1$.
\end{Lemma}
\begin{proof}
Let $c\in (C\setminus (f_2,\ldots,f_r,E))$. Then $c\in (f_1)$, let us say $c = f_1x_tx_p$. Set $I'=(f_2,\ldots,f_r,E, B\setminus\{f_1x_t,f_1x_p\})$,
$J'=I'\cap J$.  In the following exact sequence
$$0\to I'/J'\to I/J\to I/(J+I')\to 0$$
the last term is isomorphic with $(f_1)/(J+I')\cap (f_1)$ and has depth and  sdepth $\geq d+2$ because $c\not\in (J+I')$ (here it is enough that depth $\geq d+1$,
 which is easier to see). By \cite[Lemma 2.2]{R} we get $\sdepth_S I'/J'\leq d+1$. It follows that
$\depth_S I'/J'\leq d+1$ by hypothesis  and so the Depth Lemma gives  $\depth_S I/J\leq d+1$.

Now, let $ I''=(f_1,\ldots,f_r,E\setminus \{a\})$ for some $a\in E$ and $c\in C\setminus I''$. In the following exact sequence
$$0\to I''/I''\cap J\to I/J\to I/(J+I'')\to 0$$
the last term is isomorphic with $(a)/(a)\cap (J+I'')$ and has depth and sdepth $\geq d+2$ because $c\not \in J+I''$ and as above we get  $\depth_S I/J\leq d+1$.
\end{proof}

The following lemma could be seen somehow as a consequence of \cite[Theorem 1.10]{PZ},  but we give here an easy direct proof.
\begin{Lemma} \label{1} Suppose that $r=1$, let us say $I=(f)$ and $E=\emptyset$. If  $\sdepth_SI/J=d+1$, $d=\deg f$ then  $\depth_SI/J\leq d+1$.
\end{Lemma}

 \begin{proof} First assume that $d>0$. Note that $I/J\cong S/(J:f)$. We have $\sdepth_SI/J=\sdepth_S S/(J:f)$ and $\depth_SI/J=\depth_S S/(J:f)$. It is enough to treat the case $d=1$. We may assume that $x_1|f$ and  using \cite[Lemma 3.6]{HVZ} after skipping the variables of $f/x_1$
  we may reduce our problem  to the case  $d=1$.

  Therefore we may assume that $d=1$.
  If $C=\emptyset$ then $x_1x_tx_k\in J$ for all $1<t<k\leq n$ and so $(J:x_1)$ contains all squarefree monomials of degree two in  $x_t$, $t>1$, that is the annihilator of the element induced by $x_1$ in $I/J$ has dimension $\leq 2$. It follows that $\depth_SI/J\leq 2$.

 If let us say $c=x_1x_2x_3\in C$ then in the exact sequence
 $$0\to (B\setminus \{x_1x_2,x_1x_3\})/J\cap (B\setminus \{x_1x_2,x_1x_3\})\to I/J\to I/J+(B\setminus \{x_1x_2,x_1x_3\})\to 0 $$
 the last term is isomorphic with $(x_1)/(J,(B\setminus \{x_1x_2,x_1x_3\})$ and it has depth $\geq 2$ and sdepth $3$ because it has just the interval $[x_1,c]$.
 The first term is not zero since otherwise $\sdepth_SI/J=3$, which is false. Then the first term has sdepth $\leq 2$ by \cite[Lemma 2.2]{R} and so it has depth
 $\leq 2$ by\cite[Theorem 4.3]{P}. Now it is enough to apply the Depth Lemma.

 Now assume that $d=0$, that is $I=S$. Set $S'=S[x_{n+1}]$, $I'=(x_{n+1})$, $J'=x_{n+1}J$. We have $ \sdepth_{S'}I'/J'=\sdepth_{S'}S'/JS'=1+\sdepth_SS/J=2$ using \cite[Proposition 3.6]{HVZ}. From above we get $\depth_{S'}I'/J'\leq 2$ and it follows  $\depth_SI/J\leq 1$.
 \end{proof}

The following theorem extends the above lemma and \cite{PZ1}, its  proof is given in the last section.
\begin{Theorem} \label{p2} Suppose that $I \subset S$ is minimally generated by a squarefree monomial $\{f\}$,  of degree $d$ and a set $E\not =\emptyset$ of monomials of degrees $d+1$.  Assume that  $\sdepth_S I/J\leq d+1$.
 Then $\depth_S I/J\leq d+1$.
\end{Theorem}

\begin{Lemma} \label{2} Suppose that $I=(x_1,x_2)$, $E=\emptyset$. If  $\sdepth_SI/J=2$ then\\  $\depth_SI/J\leq 2$.
\end{Lemma}
\begin{proof}
 By \cite[Proposition 1.3]{PZ} we may suppose that $C\not\subset (x_2)$. Then   apply Lemma \ref{l}, its hypothesis is given by Theorem \ref{p2}.
 \end{proof}

We  need the following lemma, its proof is given in the next section.
\begin{Lemma} \label{l2} Suppose that $I \subset S$ is minimally generated by some squarefree monomials $\{f_1,f_2,f_3\}$ of degree $d$ and that $\sdepth I/J = d+1$. If there exists $c \in C \cap ((f_3) \setminus (f_1,f_2))$ then $\depth_S I/J \leq d+1$.
\end{Lemma}

\begin{Proposition} \label{p} Suppose that $I=(x_1,x_2,x_3)$, $E=\emptyset$. If  $\sdepth_SI/J=2$ then  $\depth_SI/J\leq 2$.
\end{Proposition}
\begin{proof}   By \cite[Proposition 1.3]{PZ} we may suppose that $C\not\subset (x_1,x_2)$. Then  we may  apply Lemma \ref{l2}.
\end{proof}
\begin{Remark} {\em When $J=0$ the above proposition follows  quickly from \cite{BHK} (see also \cite{HVZ}).}
\end{Remark}

\section{Case $r,d>1$}

\begin{Proposition} \label{l1} Suppose that $I \subset S$ is generated by two squarefree monomials $\{f_1,f_2\}$ of degrees $d$. Assume that  $\sdepth_S I/J\leq d+1$. Then $\depth_S I/J\leq d+1$.
\end{Proposition}
\begin{proof} We may suppose that $I$ is minimally generated by $f_1,f_2$ because otherwise  apply the Theorem \ref{p2}. Let $w$ be the least common multiple of $f_1,f_2$. First suppose that $C\not\subset (w)$. This is the case when $w\in J$, or $\deg w>d+2$, or $w\in C$ and $q>1$. Then it is enough to apply Lemma \ref{l}, the case $r=1$ being done in the Theorem \ref{p2}. If $q=1$ then $r>q$ and by \cite[Corollary 2.6]{Sh} (see also Theorem \ref{t1}) we get  $\depth_S I/J\leq d+1$.  Assume that $w\in B$. After renumbering the variables $x_i$ we may suppose that $C=\{wx_i:1\leq i\leq q\}$ and so in $B$ we have at least the elements of the form $w,f_1x_i,f_2x_i$, $1\leq i\leq q$ . Thus $s \geq 2q+1>q+2$ when $q>1$ and by \cite[Theorem 1.3]{P1} (see Theorem \ref{t}) we are done.
\end{proof}

\begin{Lemma} \label{el} Suppose that $I \subset S$ is  generated by three squarefree monomials\\
 $\{f_1,f_2,f_3\}$ of degrees $d$, $\sdepth_SI/J=d+1$ and
 let $w_{ij}$ be the least common multiple of $f_i,f_j$, $1\leq i<j\leq 3$. If $w_{12},w_{13},w_{23}\in B$ and are different then $\depth_SI/J\leq d+1$.
\end{Lemma}
\begin{proof}  After renumbering the variables $x_i$ we may assume that $f_1=x_1\cdots x_d$ and $f_2=x_1\cdots x_{d-1}x_{d+1}$. We see that $f_3$ must have $d-1$ variables in common with $f_1$ and also with $f_2$. If $f_3 \notin (x_1...x_{d-1})$ then we may suppose that $f_3=x_2...x_dx_{d+1}$ and $w_{12}=w_{13}$, which is false. It remains that $f_3 \in (x_1\cdots x_{d-1})$ so $f_3=x_1\cdots x_{d-1}x_{d+2}$. But this case may be reduced to $d=1$ which is done in Proposition \ref{p}.
\end{proof}

\begin{Lemma} \label{key} If $C\subset (w_{12},w_{13},w_{23})$ and $\sdepth_S I/J\leq d+1$ then $\depth_S I/J\leq d+1$.
\end{Lemma}
\begin{proof}
 Note that if $q<r=3$ then $\depth_S I/J\leq d+1$ by \cite[Corollary 2.6]{Sh} (see here Theorem \ref{t1}). Suppose that $q>2$.

 Now assume that all $w_{ij}\in B$. Set $C_{ij}=C\cap (w_{ij})$, $q_{ij}=|C_{ij}|$ and $B_{ij}$ the set of all $b\in B$ which divide some $c\in C_{ij}$.
If all $w_{ij}$ are equal, let us say  $w_{ij}=w$, then  after renumbering the variables $x_i$ the monomials of  $C$ have the form $wx_t$, $1\leq t\leq q$. Thus $B$ contains $w$ and $f_jx_t$ for $j\in [3]$ and $t\in [q]$. It follows that $s\geq 3q+1>q+3$ for $q>1$ and so $\depth_S I/J\leq d+1$ by \cite[Theorem 1.3]{P1}.
Then we may suppose that all $w_{ij}$ are different and we may apply Lemma \ref{el}.

Next assume that $w_{12},w_{13}\in B$ and $w_{23}\in C$. As above we can assume that $f_2=x_1\cdots x_d$, $f_3=x_3\cdots x_{d+2}$ and $f_1=x_2\cdots x_{d+1}$.  We have $C\subset C_{12}\cap C_{13}$, and  $q= q_{12}+q_{13}-1$ because $w_{23}\in C_{12}\cap C_{13}$.  As in the case of $r=2$ we have $|B_{12}|= 2q_{12}+1$ and $|B_{13}\setminus B_{12}|\geq 2q_{13}-\min\{q_{12},q_{13}\}$. It follows that $s\geq 2q+4-\min\{q_{12},q_{13}\}> q+3$, which implies $\depth_S I/J\leq d+1$ by \cite[Theorem 1.3]{P1}. Note that if  $w_{23}\in J$, or $\deg w_{23}>d+2$ then  $q= q_{12}+q_{13}$ and we get in the same way that $s\geq 2q+2-\min\{q_{12},q_{13}\}\geq q+3$. Thus $\depth_S I/J\leq d+1$ unless  $q_{12}=q_{13}=1$. The last case is false because $q>2$.

Suppose that  all $w_{ij}$ are different, $w_{12}\in B$ and $w_{23},w_{13}\in C$. We may assume that  $f_2=x_1\cdots x_d$, $f_3=x_3\cdots x_{d+2}$ and $f_1=x_2\cdots x_d\cdot x_{d+3}$. We have $q=q_{12}+2$, $B_{12}\cap B_{13}\subset \{x_{d+1}f_1,x_{d+2}f_1\}$ and so $|B_{13}\setminus B_{12}|\geq 2$. Also note that $B_{23}\cap (B_{12}\cup B_{13})\subset \{x_{d+1}f_2,x_{d+2}f_2,x_2f_3\}$  and so $|B_{23}\setminus (B_{12}\cup B_{13})|\geq 1$. It follows that
$s\geq 2q_{12}+1+2+1=2q$. If $q>3$ we get $s>q+3$ and so  $\depth_S I/J\leq d+1$ by \cite{P1}. If $q=3$ then $q_{12}=1$ and so $B_{12}=\{w_{12},x_tf_1,x_tf_2\}$ for some $x_t\not |f_1$, $x_t\not |f_2$. If $t=d+1$ or $t=d+2$ then we see that $|B_{13}\setminus B_{12}|\geq 3$ and so $s>6=r+q$, which is enough. If $t>d+3$ then $s$ is even bigger than $7$. If let us say $w_{23}\in J$, or  $\deg w_{23}>d+2$ then  $q= q_{12}+1$ and as above $s\geq 2q_{12}+1+2=2q+1>q+3$ because $q\geq 3$, which is again enough. If also $w_{13}\in J$, or  $\deg w_{13}>d+2$ then  $q= q_{12}$ and as above $s\geq 2q_{12}+1=2q+1>q+3$ because $q\geq 3$.

Suppose that   $w_{12}\in B$ and $w_{23}=w_{13}\in C$. We may assume that  $f_2=x_1\cdots x_d$, $f_3=x_3\cdots x_{d+2}$ and $f_1=x_1\cdots x_{d-1}\cdot x_{d+2}$. We have $q=q_{12}$ and  $B_{12}\supset B_{13}$. Thus $s\geq 2q_{12}+1=2q+1>q+3$ and so again $\depth_S I/J\leq d+1$.

Finally if all $w_{ij}$ are in $C$ (they must be different, otherwise $q\leq 2$ which is false) then $q=3$ , $q_{ij}=1$ and we get $s\geq 12>q+3$ which is again enough.
\end{proof}

\begin{Theorem}\label{p3}  Suppose that $I \subset S$ is  generated by three squarefree monomials $\{f_1,f_2,f_3\}$ of degrees $d$,  and $\sdepth_S I/J= d+1$. Then $\depth_S I/J\leq d+1$.
\end{Theorem}
\begin{proof} We may suppose that $I$ is minimally generated by $f_1,f_2,f_3$ because otherwise apply Proposition \ref{l1}.
If $C\not\subset (w_{12},w_{13},w_{23})$ then apply Lemma \ref{l2}.
 Thus we may suppose that $C\subset (w_{12},w_{13},w_{23})$ and we may apply Lemma \ref{key}.
\end{proof}

\section{Proof of Lemma \ref{l2}}

Let $c= f_3x_{i_3}x_{j_3}$ and  set $I' = (f_1,f_2,B \setminus \{f_3x_{i_3},f_3x_{j_3}\}), J' = I' \cap J$. Consider the following exact sequence

$$0\to I'/J'\to I/J \to I/(I' + J)\to 0.$$

The last term has $\sdepth =  d+2$ so by \cite[Lemma 2.2]{R} we get that the first term has $\sdepth \leq d+1$. If $\depth I'/J' \leq d+1$ then by Depth Lemma we are done. It is {\bf enough  to show that} $\sdepth_SI'/J' = d+1$ {\bf implies}  $\depth_SI'/J' \leq d+1$, or directly $\depth_SI/J \leq d+1$.
Note that if $\sdepth_SI'/J'=d$ then $\depth_SI'/J'=d$ by Theorem \ref{P}.
Let $B'$, $C'$, $E'$ be similar to $B$, $C$, $E$ in the case of $I'/J'$.

We see that $E' \subset (f_3)$.
We may suppose that $C'\subset ((f_1)\cap (f_2))\cup (E')$ and $E' \neq \emptyset$, otherwise apply Lemma \ref{l} with the help of Theorem \ref{p2}.

Set $I'_E = (f_1,f_2), J'_E = I'_E \cap J'$ and for all $i\in [n]\setminus \supp f_1$ such that $f_1x_i\in B'\setminus (f_2)$ set $I'_i=(f_2,B\setminus\{f_1x_i\})$,   $J'_i = I'_i \cap J'$. We may suppose that $\sdepth_S I'_E / J'_E \geq d+2$ and $\sdepth_SI'_i/J'_i\geq d+2$. Indeed, otherwise one of the left terms from the following exact sequences
$$0\to I'_E/J'_E\to I'/J'\to I'/I'_E+J'\to 0,$$
$$0\to I'_i/J'_i\to I'/J'\to I'/I'_i+J'\to 0,$$
have depth $\leq d+1$ by Proposition \ref{l1} and Theorem \ref{p2}. With the Depth Lemma we get $\depth_SI'/J'\leq d+1$ since the right terms above have depth $\geq d+1$.
Let ${\mathcal P}_E$, ${\mathcal P}_i$ be partitions of $I'_E/J'_E$, $I'_i/J'_i$ with sdepth $d+2$.
	We may choose ${\mathcal P}_E$ and ${\mathcal P}_i$ such that each interval starting with a squarefree monomial of degree $\leq d+1$ ends with a monomial from $C'$.

Our goal is mainly to {\bf reduce our problem to the case when} $w_{13},w_{12}\in B'\cup C'$.
{\bf Case 1}\ \ $C'\not \subset (f_1,f_3)\cap (f_2,f_3)$

Let for example  $c=f_1x_ux_v \in C'\setminus (f_2,f_3)$, set $I''=(f_2,B' \setminus \{f_1x_u,f_1x_v\}),J'' = I'' \cap J'$ and consider the exact sequence:

$$0\to I''/J''\to I'/J' \to I'/(I'' + J')\to 0.$$

The last term has sdepth  $d+2$ so by \cite[Lemma 2.2]{R} we see that the first term has $\sdepth \leq d+1$. Using Theorem \ref{p2} we have $\depth_S I''/J'' \leq d+1$ and then by the Depth lemma we get $\depth_S I'/J' \leq d+1$ ending Case 1.

Let $f_1= x_1...x_d$ , in ${\mathcal P}_E$ we have the intervals $[f_1,c_1],[f_2,c_2]$ and so at least one of $c_1,c_2$, let us say   $c_1 = f_1x_ix_j$,  is not a multiple of  $w_{12}$.
In ${\mathcal P}_i$ we have the interval $[b,c_1]$ for some $ b\in E'$, otherwise replacing the interval $[f_1x_j,c_1]$ or the interval $[c_1,c_1]$ with the interval $[f_1,c_1]$ we get a partition ${\mathcal P}$ for $I'/J'$ with $\sdepth = d+2$.

 {\bf Case 2} \ There exists $t\in [n]$, $t\not\in \supp f_1\cup \{i\}$
 such that ${\mathcal P}_i$ contains the interval $[f_1x_t,f_1x_tx_i]$, or $[f_1x_tx_i, f_1x_tx_i]$.

 In this case changing in ${\mathcal P}_i$ the hinted interval with $[f_1,f_1x_tx_i]$ we get a partition of $I'/J'$ with sdepth $\geq d+2$ which is false.

As we have seen  above  we may suppose that
  in ${\mathcal P}_i$ there exists an interval $[b,c_1]$ with $c_1 \in (f_1)\cap (E')\subset (w_{13})$. {\bf It follows that} $w_{13}\in B'\cup C'$. We may assume that if $w_{13} \in B'$ then $x_i\not | w_{13}$, otherwise change $i$ by $j$. Thus $c_1=x_iw_{13}$ or $c_1=w_{13}$.  If $C' \cap (f_1x_i) = \{c_1\}$ then in ${\mathcal P}_j$ we  have the interval $[f_1x_i,c_1]$, that is we are  in Case 2. Then there exists another monomial $c' \in C'\cap(f_1x_i)$. We may suppose that $[c',c']$ is not  in ${\mathcal P}_i$, because otherwise we are in Case 2. If we have $[u,c']$ in ${\mathcal P}_i$ for some $u\in E'$ then $c'\in (w_{13})$ and so $c'=c_1$ if $w_{13}\in C'$, otherwise $c'=x_iw_{13}=c_1$ because $x_i\not|w_{13}$. Contradiction!   Then in ${\mathcal P}_i$  we have the interval $[f_2,c']$ or the interval $[f_2x_k,c']$ for some $k$. Thus  $c' \in (w_{12})$ {\bf and so} $w_{12}\in B'\cup C'$. Note that $w_{12}\not =w_{13}$ because $c_1\in (w_{13})\setminus (f_2)$.

{\bf Case 3}\ \
 $w_{12}, w_{13} \in C'$.

In this case $c_1=w_{13}$, $c'=w_{12}$ and so $f_2,f_3 \in (x_i)$. Then in ${\mathcal P}_i$ we have the interval $[f_1x_j,f_1x_jx_u], u \neq i$ and $f_1x_jx_u \notin (f_2,f_3)$ because $f_1x_jx_u \notin (x_i)$, that is we are in Case 1.

{\bf Case 4}\ \ $w_{12} \in B', w_{13} \in C' $.

  Thus $c_1=w_{13}$. We may assume that $w_{12} = x_1...x_{d+1},f_2 = x_2...x_{d+1}, i \neq d+1 \neq j$ and $c'=x_1...x_{d+1}x_i$. We also see that $f_3 \in (x_ix_j)$ because  $c_1=w_{13}$.
In ${\mathcal P}_i$ we have the interval $[f_1x_j,f_1x_jx_u], u \neq i$. If $u \neq d+1$ then $f_1x_jx_u \notin (f_2,f_3)$, that is we are in Case 1. Otherwise $u=d+1$, and so $x_jw_{12}\in C'$, in particular $f_2x_j\in B'$. We see that in ${\mathcal P}_i$ we can have either $w_{12} \in [f_2,c']$ or there exists an interval $[w_{12},w_{12}x_k]$. If $k=j$ then $w_{12}x_k$  is the end of the interval starting with $f_1x_j$, which is false. If $k=i$ then we are in Case 2. Thus $i \neq k \neq j$.

  When in ${\mathcal P}_i$  there exists the interval  $[w_{12},w_{12}x_k]$ then there exists also the interval $[f_1x_k,f_1x_kx_t]$. If $f_1x_kx_t\in (f_2)$ then $t=d+1$ and so $f_1x_kx_t=x_kw_{12}$ which is not possible because $x_kw_{12}$ is in $[w_{12},x_kw_{12}]$. If $f_1x_kx_t\in (f_3)$ then $\{k,t\}=\{i,j\}$ which is not possible since $k\not \in \{i,j\}$.  Then $f_1x_kx_t \notin (f_2,f_3)$, that is we are in Case 1. It remains the case when $w_{12}$ is in  the interval $[f_2,c']$. In ${\mathcal P}_i$ we have an interval $[f_2x_j,f_2x_lx_j]$ for some $l$. If $f_2x_jx_l\in (f_1)$ then $l=1$ and so  $f_2x_jx_l=x_jw_{12}$ which is already the end of the interval starting with $f_1x_j$. Contradiction ! Thus  $f_2x_lx_j \in (f_3)$, otherwise we are in Case 1. We get $l=i$ and changing $[f_2,c'],[f_2x_j,f_2x_ix_j]$ with $[f_2,f_2x_ix_j],[w_{12},c']$ we arrive in Case 2.

{\bf Case 5}\ \ $w_{12} \in C', w_{13} \in B'$.

 Thus we may assume that $w_{12}= c' = x_1...x_{d+1}x_i, f_2 = x_3...x_{d+1}x_i$.
 As $c_1\in (w_{13})$ we have $w_{13}\in \{f_1x_i,f_1x_j\}$.
  If $w_{13} = f_1x_i$ then in ${\mathcal P}_i$ we have an interval $[f_1x_j,f_1x_jx_u]$. If $f_1x_jx_u\in (f_2)$ then $u=i$. Also if $f_1x_jx_u\in (f_3)$ we get $ f_1x_j x_u\in (w_{13})$ and we get again $u=i$, that is we are in Case 2. Thus $f_1x_jx_u \notin (f_2,f_3)$ and we arrive in Case 1.

   Then, we may suppose that $w_{13} = f_1x_j$.
Since $f_1x_{d+1}|c'$ we see that $f_1x_{d+1}\in B'$.
In ${\mathcal P}_i$ we can  have the interval $[f_1x_{d+1},f_1x_{d+1}x_m]$. If $f_2x_{d+1}x_m\in (f_2)$ then $m=i$, that is we are in Case 2. Then  $f_2x_{d+1}x_m\in (f_3)$ because otherwise we are in Case 1. It follows that $m=j$ and we have $[f_1x_{d+1}, f_1x_{d+1}x_j]$ in ${\mathcal P}_i$. Then the  interval $[f_1x_j,f_1x_jx_p]$ existing in ${\mathcal P}_i$ has $p\not =j$ and also $p\not =i$ because otherwise we are in Case 2. Thus we must
 also have an interval $[f_1x_p,f_1x_px_k]$ with  $k\not =j$ and also $k\not= i$, otherwise we are in Case 2. Then  $f_1x_px_k \notin (f_2,f_3)$, that is we are in Case 1.

{\bf Case 6}  \ \ $w_{12},w_{13} \in B'$.

 We may assume that $w_{12} = x_1...x_{d+1}, f_2 = x_2...x_{d+1}$ and $c' = x_1...x_{d+1}x_i$. If $w_{23} \in B'$ then all $w_{ij}$ are different  and by Lemma \ref{el} we get $\depth_SI/J\leq d+1$. Thus we may suppose that
$w_{23} \in C'$. We may choose $f_3 = x_1x_3...x_dx_i$ or $f_3 = x_1x_3...x_dx_j$. If $f_3 = x_1x_3...x_dx_i$ then in ${\mathcal P}_i$ we have as above  the interval $[f_1x_j,f_1x_{d+1}x_j]$. Indeed, if we have $[f_1x_j,f_1x_mx_j]$ then $f_1x_mx_j\not\in (f_3)$ and so $f_1x_mx_j\in (f_2)$, otherwise we are in Case 1. It follows $m=d+1$. As $f_2x_j|x_jw_{12}=f_1x_{d+1}x_j$ we get $f_2x_j\in B'$.  Let $[f_2,f_2x_jx_k]$ or $[f_2x_j,f_2x_jx_k]$ be the existing interval of ${\mathcal P}_i$ containing $f_2x_j$. Note that $f_2x_jx_k\not \in (f_3)$ and if  $f_2x_jx_k\in (f_1)$ then $f_2x_jx_k=x_jw_{12}$ which appeared already in the previous interval. Thus
 $f_2x_jx_k \notin (f_1,f_3)$, that is we are in Case 1.

 It remains that $f_3=x_1x_3...x_dx_j$ and, as before, we have in ${\mathcal P}_i$ the interval $[f_1x_j,f_1x_{d+1}x_j]$. We see then $f_2x_j\in B'$ and  we  must have also an interval  $[f_2,f_2x_jx_k]$ or $[f_2x_j,f_2x_jx_k]$. If $f_2x_jx_k\in (f_1)\cup (f_3)$ then we get $k=1$ and so $f_2x_jx_k=x_jw_{12}$ which appeared in the previous interval. It follows that $f_2x_jx_k \notin (f_1,f_3)$, that is we are in Case 1.

\section{Proof of Theorem \ref{p2}}

Suppose that $E\not =\emptyset$ and $s\leq q+1$. We may assume that $|B\setminus E|\geq 2$ because otherwise $\depth_SI/J\leq d+1$ since the element induced by $f$ in $I/J$ is annihilated by all variables but one and those from $\supp f$.
For   $b=fx_{i}\in B$ set $I_b=(B\setminus\{b\})$, $J_b=J\cap I_{b}$. If $\sdepth_SI_{b}/J_{b}\geq d+2$ then let ${\mathcal P}_b$ be a partition on $I_{b}/J_{b}$ with sdepth $d+2$. We may choose ${\mathcal P}_b$ such that each interval starting with a squarefree monomial of degree $d$, $d+1$ ends with a  monomial of $C$. In $ {\mathcal P}_{b}$ we  have some intervals for all $b' \in B\setminus\{b\}]$ an interval $[b',c_{b'}]$. We define $h:B\setminus\{b\}\to C$ by  $b'\to c_{b'}$. Then $h$ is an injection and $|\Im h|= s-1\leq q$ (if $s=1+q$ then $h$ is a bijection).
We may suppose that all intervals   of $ {\mathcal P}_{b}$ starting with a monomial $v$ of degree $\geq d+2$ have the form $[v,v]$.

\begin{Lemma}\label{ml} Suppose that the following conditions hold:
\begin{enumerate}
\item{}  $s\leq q+1$,
\item{}  $\sdepth_SI_{b}/J_b\geq d+2$, for a $ b\in B\cap (f)$,
\item{} $C\subset ((f)\cap (E)) \cup (\cup_{a,a'\in E, a\not =a'}(a)\cap (a'))$.
\end{enumerate}
Then either  $\sdepth_SI/J\geq d+2$, or there exists a nonzero ideal $I'\subsetneq I$ generated by a subset  of $\{f\}\cup B$ such that $\sdepth_S I'/J'\leq d+1$ for  $J'=J\cap I'$ and $\depth_SI/(J,I')\geq d+1$.
\end{Lemma}
\begin{proof}
Consider  $h$ as above for a partition ${\mathcal P}_b$  with sdepth $d+2$ of $I_b/J_b$ which exists by (2). A sequence $a_1, \ldots , a_k$ is called a {\em path} from $a_1$ to $a_k$ if $a_i\in B \setminus \{b\}$, $i \in [k]$,
$a_i \not = a_j$
for $1 \leq i < j \leq k$, $a_{i+1}|h(a_i)$ for $1 \leq i < k$, and $h(a_i)\not\in (b)$ for $1 \leq i < k$.
This path is {\em bad} if $h(a_k) \in (b)$ and it is {\em maximal} if all divisors from B of $h(a_k)$ are in
$\{b, a_1,\ldots , a_k\}$. If $a = a_1$ we say that the above path {\em starts with} $a$. Since $|B\setminus E|\geq 2$ there
exists $a_1 \in B \setminus \{b\}$. Set $c_1 = h(a_1)$. If $c_1 \in (b)$ then the path $\{a_1\}$ is maximal and
bad. By recurrence choose if possible $a_{p+1}$ to be a divisor from $B$ of $c_p$  which is not
in $\{b, a_1,\ldots , a_p\}$ and set $c_p = h(a_p)$, $p \geq 1$. This construction ends at step $p = e$ if
all divisors from $B$ of $c_{e-1}$ are in $\{b, a_1, \ldots , a_{e-1}\}$. If $c_i \not \in (b)$ for $1 \leq i < e - 1$ then
$\{a_1, \ldots , a_{e-1}\}$ is a maximal path. If $c_{e-1} \in (b)$ then this path is also bad. We have
two cases:

1) there exist no maximal bad path starting with $a_1$,

2) there exists a maximal bad path starting with $a_1$.

In the first case, set $T_1 = \{b' \in B : \mbox{there\ exists\ a \  path}\ a_1, \ldots , a_k \ \mbox{with}\ a_k = b'\}$,
$G_1 = B \setminus T_1$ and $I'_1 = (f, G_1)$, $I''_1 = (G_1)$, $J'_1 = I'_1 \cap J$, $J''_1 = I''_1 \cap J$. Note that $I''_1\not = 0$
because $b \in I''_1$. Consider the following exact sequence
$$0 \to I'_1/J'_1 \to I/J \to I/(J, I'_1) \to 0.$$
If $T_1 \cap (f) = \emptyset$ then the last term has depth $\geq d + 1$ and sdepth $\geq d + 2$ using the
restriction of ${\mathcal P}_b$ since $ h(b') \not \in I'_1$, for all $b' \in T_1$. If the first term has sdepth $\geq d + 2$
then by \cite[Lemma 2.2]{R} the middle term has sdepth $\geq d + 2$. Otherwise, the first
term has sdepth $\leq d + 1$ and we may take $I' = I'_1$.

If let us say $a \in (f)$ for some
$a \in  T_1$ then in the following exact sequence
$$0 \to I''_1/J''_1 \to I/J \to I/(J, I''_1) \to 0$$
the last term has sdepth $\geq d + 2$ and depth $\geq d + 1$ since $h(a) \not \in I''_1$
and we
may substitute the interval $[a, h(a)]$ from the restriction of ${\mathcal P}_b$ to $(T_1)$ by $[f, h(a)]$,
the second monomial from $[f, h(a)]\cap B$ being also in $T_1$. As above we get either
$\sdepth_SI/J = d + 2$, or $\sdepth_SI''_1 /J''_1 \leq  d + 1$, $\depth_SI/(J, I''_1) \geq d + 1$.

In the second case, let $a_1, \ldots , a_{t_1}$ be a maximal bad path starting with $a_1$. Set
$c_j = h(a_j)$, $j \in [t_1]$. Then $c_{t_1} = bx_{u_1}$ for some $u_1$ and let us say $b = fx_i$. If $a_{t_1} \in (f)$ then
changing in ${\mathcal P}_b$ the interval $[a_{t_1}, c_{t_1}]$ by $[f, c_{t_1}]$ we get a partition on $I/J$ with sdepth
$d + 2$. Thus we may assume that $a_{t_1} \in E$. If $fx_{u_1} \in \{a_1, \ldots, a_{t_1-1}\}$, let us say
$fx_{u_1} = a_v$, $1 \leq v < t_1$ then we may replace in ${\mathcal P}_b$ the intervals $[a_p, c_p]$, $v \leq p \leq t_1$
with the intervals $[a_v, c_{t_1}]$, $[a_{p+1}, c_p]$, $v \leq  p < t_1$. Now we see that we have in ${\mathcal P}_b$ the
interval $[fx_{u_1}, fx_ix_{u_1}]$ and switching it with the interval $[f, fx_ix_{u_1}]$ we get a partition
with sdepth $\geq d + 2$ for $I/J$.
Thus we may assume that $fx_{u_1} \not \in \{  a_1,\ldots, a_{t_1}\}$. Now set $a_{t_1+1} = fx_{u_1}$. Let
$a_{t_1+1}, \ldots , a_k$ be a path starting with $a_{t_1+1}$ and set $c_j = h(a_j)$, $t_1 < j \leq k$. If $a_p = a_v$
for $v < t_1$, $p > t_1$ then change in ${\mathcal P}_b$ the intervals $[a_j, c_j]$, $v \leq j \leq p$ with the intervals
$[a_v, c_p]$, $[a_{j+1}, c_j]$, $v \leq j < p$. We have in ${\mathcal P}_b$ an interval $[fx_{u_1}, fx_ix_{u_1}]$ and switching it
to $[f, fx_ix_{u_1}]$ we get a partition with sdepth $\geq d + 2$ for $I/J$. Thus we may suppose
that in fact $a_p\not\in  \{b, a_1, \ldots , a_{p-1}\}$ for any $p > t_1$ (with respect to any path starting
with $a_{t_1+1}$). We have again two subcases:

$1')$ there exist no maximal bad path starting with $a_{t_1+1}$,

$2')$ there exists a maximal bad path starting with $a_{t_1+1}$.

In $1')$ set $T_2 = \{b' \in B : \mbox{there\ exists\ a \  path}\ a_{t_1+1}, \ldots , a_k \ \mbox{with}\ a_k = b'\}$, $ G_2 = B \setminus T_2$
 and $I'_2 = (f, G_2)$, $I''_2 = (G_2)$, $J'_2 = I'_2 \cap J$, $J''_2 = I''_2 \cap J$. As above,
we see that if $T_2 \cap (f) = \emptyset$ then we may take $I' = I'_2$ and if $T_2 \cap (f)\not =\emptyset$ then $I' = I''_2
$ works.

In the second case, let $a_{t_1+1}, \ldots , a_{t_2}$ be a maximal bad path starting with $a_{t_1+1}$
and set $c_j = h(a_j)$ for $j > t_1$. As we saw the whole set $\{a_1, \ldots , a_{t_2}\}$ has different
monomials. As above $c_{t_2} = bx_{u_2}$ and we may reduce to the case when $fx_{u_2} \not\in
\{a_1, \ldots , a_{t_1}\}$. Set $a_{t_2+1} = fx_{u_2}$ and again we consider two subcases, which we treat
as above. Anyway after several such steps we must arrive in the case $p = t_m$ when
$b|c_{t_m}$ and again a certain $fx_{u_m}$ is not among $\{a_1, \ldots , a_{t_m}\}$ and taking $a_{t_m+1} = fx_{u_m}$
there exist no maximal bad path starting with $a_{t_m+1}$. This follows since we may
reduce to the case when the set $\{a_1,\ldots  , a_{t_m}\}$ has different monomials and so the
procedures should stop for some m. Finally, using
 $T_m = \{b' \in B : \mbox{there\ exists\ a \  path}\ a_{t_m+1}, \ldots , a_k\  \mbox{with}\ a_k = b'\}$
as $T_1$ above we are done.
\end{proof}

Now Theorem \ref{p2} follows from the next proposition, the case $s>q+1$ being a consequence of \cite{P1} (see here Theorem \ref{t}).
\begin{Proposition}  Suppose that $I \subset S$ is minimally generated by a  squarefree monomial $f$ of degree $d$,  and a  set $E$ of squarefree monomials of degrees $\geq d+1$.   Assume that
 $\sdepth_S I/J= d+1$ and  $s\leq q+1$.
Then $\depth_S I/J \leq d+1$.
\end{Proposition}
\begin{proof} Apply induction on $|E|$, the case $E=\emptyset $ follows from Lemma \ref{1}. Suppose that $|E|>0$.
We may assume  that  $E$ contains just monomials of degrees $d+1$  by \cite[Lemma 1.6]{PZ}. Using Theorem \ref{P} and induction on $|E|$ apply Lemma \ref{l}. Thus we may suppose that  $C\subset ((f)\cap (E)) \cup (\cup_{a,a'\in E, a\not =a'}(a)\cap (a'))$.

Let $b\in (B\cap (f))$ and $I'_b=(B\setminus \{b\})$. Set $J'_b=I'_b\cap J$.  Clearly $b\not \in I'_b$.
  As in Case 1 from the previous section we see that if $\sdepth_S I'_b/J'_b\leq d+1$ then $\depth_S I'_b/J'_b\leq d+1$ by Theorem \ref{P} and so $\depth_S I/J\leq d+1$ by the Depth Lemma. Thus we
  may suppose that $\sdepth_SI'_b/J'_b\geq  d+2$. Applying Lemma \ref{ml} we get either $\sdepth_SI/J\geq d+2$ contradicting our assumption, or there
 exists a nonzero ideal $I'\subsetneq I$ generated by a subset  of $\{f\}\cup B$ such that $\sdepth_S I'/J'\leq d+1$ for  $J'=J\cap I'$ and $\depth_SI/(J,I')\geq d+1$. In the last case we see that  $\depth_S I'/J'\leq d+1$ by induction hypothesis on $|E|$ and so $\depth_SI/J\leq d+1$ by the Depth Lemma applied to the following exact sequence
$$0\to I'/J'\to I/J\to I/(J,I')\to 0.$$
\hfill\ \end{proof}

\noindent\textbf{Acknowledgement}\textit{.} Research partially supported by  grant ID-PCE-2011-1023 of Romanian Ministry of Education, Research and Innovation.

Andrei Zarojanu was supported by the strategic grant POSDRU/159/1.5/S/137750, "Project Doctoral and Postdoctoral programs support for increased competitiveness in Exact Sciences research" cofinanced by the European Social Found within the Sectorial Operational Program Human Resources Development $2007 - 2013$.

\smallskip\noindent 

\address{Dorin Popescu,\\ Simion Stoilow Institute of Mathematics of Romanian Academy, Research unit 5,
 P.O.Box 1-764, Bucharest 014700, Romania\\
E-mail:{ dorin.popescu @ imar.ro}
}

\address{Andrei Zarojanu, \\ Faculty of Mathematics and Computer Sciences, University
of Bucharest,\\ Str. Academiei 14, Bucharest, Romania,  and}
\address{Simion Stoilow Institute of Mathematics of Romanian Academy, Research group of the project  ID-PCE-2011-1023,
 P.O.Box 1-764, Bucharest 014700, Romania\\
E-mail:{ andrei\_zarojanu @ yahoo.com}
}

\end{document}